\documentclass[11pt]{article}

\usepackage{amsmath,amssymb,amsthm}
\usepackage{color}
\usepackage{graphicx}
\newtheorem{theorem}{Theorem}

\newtheorem{definition}{Definition}
\newtheorem{lemma}{Lemma}
\newtheorem{cor}{Corollary}
\newtheorem{proposition}{Proposition}
\newtheorem{question}{The question}
\newtheorem{remark}{Remark}

\def \beq{ \begin{equation}}
\def \eeq{\end{equation}}

\title{Three body relative equilibria on $\mathbb{S}^2$}
\date{}  
\begin{document}
	\maketitle
	\author{\begin{center}
	{ Toshiaki~Fujiwara$^1$, Ernesto P\'{e}rez-Chavela$^2$}\\	
		\bigskip
	   $^1$College of Liberal Arts and Sciences, Kitasato University,       Japan. fujiwara@kitasato-u.ac.jp\\
	    $^2$Department of Mathematics, ITAM, M\'exico.\\ ernesto.perez@itam.mx
	\end{center}
	
\date{}

\bigskip

\begin{abstract}
We study relative equilibria
($RE$ in short) 
for three-body problem
on $\mathbb{S}^2$,
under the influence of a general potential which only depends on 
$\cos\sigma_{ij}$ where $\sigma_{ij}$ are the mutual angles
among the masses.
Explicit conditions for 
masses $m_k$ and  $\cos\sigma_{ij}$
to form relative equilibrium are shown.
Using the above conditions,
we study  the equal masses case
under the cotangent potential. 
We show the existence of scalene and isosceles Euler
$RE$, and isosceles 
and equilateral Lagrange $RE$.
\end{abstract}


{\bf Math. Subject Class 2020:} 70F07, 70F10, 70F15}

\section{Introduction}
Consider a point $q=(X,Y,Z)^T$ 
on $\mathbb{S}^2$, $|q|^2=R^2$.
Here, ${}^T$ represents the transpose.
In this paper, when we use the matrix notation, 
any vector is represented as a column vector.
For two vectors $a$ and $b$,
$(a\cdot b)=a^T b$ and  $a\times b$
represents the inner and the vector product respectively.
In spherical coordinates $(R,\theta,\phi)$, 
$q$ is represented by
$q=R(\sin\theta\cos\phi,\sin\theta\sin\phi,\cos\theta)^T$.
The arc angle $\sigma_{ij} \in (0,\pi)$
is  the angle between the two points $i$ and $j$
as seen from the center of $\mathbb{S}^2$.
Then, $(q_i\cdot q_j)=R^2\cos\sigma_{ij}$.
Using spherical coordinates,
$\cos\sigma_{ij}$ is given by
\begin{equation}
\label{fundamentalrelation}
\cos\sigma_{ij}
=\cos\theta_i\cos\theta_j+\sin\theta_i\sin\theta_j\cos(\phi_i-\phi_j).
\end{equation}

The $n$-body problem on $S^2$ is defined by the
Lagrangian, 
\begin{equation}\label{theLagrangian}
L=K+V
\end{equation}
 with
\begin{equation*}
K=R^2\sum\nolimits_k\frac{m_k}{2}
		\left(\dot{\theta}_k^2+\sin^2(\theta_k)\dot{\phi}_k^2\right),
\quad
V=\sum\nolimits_{i<j}\frac{m_i m_j}{R} U(\cos\sigma_{ij}).
\end{equation*}
Dot on symbols represents the time derivative,
the indexes $i,j,k$ run for $1,2,3,\dots, n$,
and $m_k$ are the masses.
The equations of motion are derived from the Lagrangian.

A relative equilibrium 
($RE$)
on the Euclidean plane, is a solution of a system of $n$--positive point masses, where each mass is rotating uniformly around the center of mass of the system  with the same angular velocity.  
The mutual distances among the masses remain constant along the motion. The masses behave as if they belong to a rigid body. For the classical planar
Newtonian $3$--body problem, it is well known that there are five classes of relative equilibria, three collinear or Euler $RE$ and two equilateral triangle or Lagrange $RE$ \cite{Euler, Moeckel}. 
The configuration
of the masses in a $RE$ is called a {\it central configuration}, we obtain the respective relative equilibrium by taking a particular uniform rotation through the center of mass \cite{Wintner}.

Similarly,
a relative equilibrium
on $\mathbb{S}^2$
is a solution of the equations of motion 
where each mass is rotating uniformly
around an axis, with the same angular velocity, 
 like a rigid body.

\begin{definition}[Relative equilibrium]\label{def1}
A relative equilibrium on $\mathbb{S}^2$ is a solution 
of the equations of motion given by the Lagrangian \eqref{theLagrangian},
which satisfies $\dot q_k=\Omega \times q_k$ for $k=1,2,\dots, n$ where $\Omega$ is a constant vector in $\mathbb{R}^3$.
\end{definition}

The constant vector $\Omega$ is called {\it the angular velocity}.
It is obvious from the definition that
$|q_k|$, 
$(q_i\cdot q_j)$, $|q_i-q_j|$ 
and $(q_k \cdot \Omega)$ are constant in time.

Unlike to the Euclidean plane, the $RE$ on $\mathbb{S}^2$ contain the following special solution.

\begin{definition}[Fixed point]\label{fixed-point}
A fixed point is a $RE$
with $\dot q=0$.
\end{definition}

\begin{definition}[Collinear and non-collinear]\label{coll-and non}
A collinear $RE$ 
is a $RE$ where all $n$--bodies are on the same geodesic. If this is not the case, we call it non-collinear $RE$.
For $n=3$, a collinear $RE$ is called 
Euler $RE$ ($ERE$), the non collinear $RE$ are called 
Lagrange $RE$ ($LRE$).
\end{definition}
 
Our main goal in this paper
is to give a systematic method
to find $RE$ on $\mathbb{S}^2$ for $n=3$.
More specifically,
we will give a clear answer to the following question.
\begin{question}\label{question}
For given positive masses $m_1,m_2,m_3$ and given arc angles $\sigma_{12}, \sigma_{23}, \sigma_{31}$, can this shape be a $RE$?
If so, what is the appropriate $\Omega$?
\end{question}

In the Euclidean planar $n$-body problem,
the rotation center for a $RE$ 
is obviously the center of mass.
But there are no center of mass 
first integral on $\mathbb{S}^2$
(see for instance \cite{Borisov1, Borisov2, Diacu-EPC1, Diacu1, Diacu3, EPC1, Shchepetilov2} and the references therein).

Fortunately
we can apply the knowledge for the rigid body,
since $RE$ does not change its shape
as if it were a rigid body.
In the rigid body problem
\cite{Goldstein, Hestenes, LandauLifshitz, Marsden, Routh},
it is well known that if the rotation axis is constant in time then the rotation axis is 
one of the principal axes 
(eigenvectors) of the inertia tensor.
Routh called such a rotation axis
``a permanent axis of rotation'' \cite{Routh}.

Starting from this well known fact,
we will give a complete answer to the above question.
In this paper we will show:
\begin{itemize}
\item Explicit equations (necessary and sufficient 
conditions) 
for $m_k$ and $\cos\sigma_{ij}$ to form a $RE$.
\item The suitable $\Omega$ is determined for the solution 
of the above condition.
Explicit expression for $\cos\theta_k$ is given.
	(The set $\cos\sigma_{ij}$ and $\cos\theta_k$ determine
	$\cos(\phi_i-\phi_j)$.
	Therefore, the configuration $\{\theta_k,\phi_i-\phi_j\}$ 
	of a $RE$ is determined.)
\end{itemize}

After the introduction, the paper is organized as follows: In Section \ref{prelimsCI}
we obtain the expression for the angular momentum,
and the inertia tensor for $RE$.
In this section, we show a special expression of inertia tensor $J$ 
for the three-body problem,
which will turn out to be useful
to investigate three-body $LRE$.
In Section \ref{secEqOfMotion}
we get the equations of motion for $n$ particles with positive masses moving on $\mathbb{S}^2$ in spherical coordinates. 
The Sections \ref{sec:Euler} and \ref{sec:Lagrange} are devoted to give necessary and sufficient conditions to form
$ERE$ and  $LRE$ on $\mathbb{S}^2$
respectively. In both cases we give a clear answer to the question \ref{question}.

In order to have concrete examples to show how  
the conditions work,
in Section \ref{sec:curved}, 
we study the
 three-body problem on $\mathbb{S}^2$ with equal masses moving
under the influence of the cotangent potential
(see for instance \cite{Borisov2, Diacu-EPC1, Diacu1}).
We show some new families of $ERE$ and $LRE$ in this problem. 
Finally in Section \ref{conclusions}, we summarize our results and state some final remarks.
 

\section{Angular momentum, and the inertia tensor}\label{prelimsCI}
From here on, for simplicity we take $R=1$. 

\subsection{For the $n$-body problem on $\mathbb{S}^2$}
Since the Lagrangian \eqref{theLagrangian} is  invariant under the action of the group $SO(3)$, 
the angular momentum 
$c
= (c_x, c_y, c_z)^T =\sum_\ell m_\ell q_\ell \times \dot q_k$
is a first integral \cite{Diacu-EPC1, Diacu1}.
For $RE$,
\begin{equation}
c
=\sum\nolimits_\ell m_\ell q_\ell \times (\Omega\times q_\ell)
=\sum\nolimits_\ell m_\ell 
	\left(\Omega-(q_\ell\cdot \Omega)q_\ell\right)
=I\Omega,
\end{equation}
where $I$ is the inertia tensor defined by
\begin{equation}
I=M-\sum\nolimits_\ell m_\ell q_\ell q_\ell^T,
\end{equation}
and $M=\sum_\ell m_\ell$ is the total mass.

In the matrix form,
$q_\ell q_\ell^T$ is a 3 by 3 matrix,
actually,
applying the usual matrix product rule
for two vectors $a=(a_1,a_2,a_3)^T$ and $b=(b_1,b_2,b_3)^T$,
the product of column vector $a$ and row vector $b^T$ in this 
order yields
\begin{equation}
ab^T
=\left(\begin{array}{c}
a_1 \\a_2 \\a_3\end{array}
\right)(b_1, b_2,b_3)
=\left(\begin{array}{ccc}
a_1b_1 & a_1b_2 & a_1b_3 \\
a_2b_1 & a_2b_2 & a_2b_3 \\
a_3b_1 & a_3b_2 & a_3b_3
\end{array}\right).
\end{equation}
Since $a^Tb=(a_1b_1+a_2b_2+a_3b_3)=(a\cdot b)$, we have that 
for any vectors $a,b,c$,
$(ab^T)c=a(b^Tc)=a(b\cdot c)$ and $c^T(ab^T)=(c\cdot a)b^T$.
Sometimes, $ab^T$ is called ``dyadic''.
In the matrix addition or subtraction, $M$ should be understood
to be $ME$, where $E$ is the 3 by 3 identity matrix,
as in the usual convention.

Now, for $\Omega\ne 0$ we obtain
\begin{equation}\label{cForOmegaDirection}
(\Omega\cdot c)
=\Omega^T I \Omega
=M|\Omega|^2-\sum\nolimits_\ell m_\ell (q_\ell\cdot \Omega)^2
=\sum\nolimits_\ell m_\ell |q_\ell \times \Omega|^2>0.
\end{equation}
Because if it is zero, then any mass should be at
either one of the two points $\pm \Omega/|\Omega|$.
For $n\ge 3$, 
there are at least one pair with $\sigma_{ij}=0$,
which we exclude.
The conservation of $c$ yields
\begin{equation}\label{cForPerpendicularToOmegaDirection}
\Omega \times c
=-\sum\nolimits_\ell m_\ell \Omega \times q_\ell (q_\ell\cdot \Omega)
=-\sum\nolimits_\ell m_\ell \dot q_\ell(q_\ell\cdot \Omega)
=\dot c
=0.
\end{equation}
The equation
\eqref{cForPerpendicularToOmegaDirection}
concludes that $c=I\Omega$ and $\Omega$ are parallel.
Namely, $\Omega$ is one of the eigenvectors of $I$,
\begin{equation}\label{theEigenvalueI}
I\Omega=\lambda \Omega.
\end{equation}
This is the well known result stated in the introduction.

\begin{lemma}\label{rotationAxis}
The rotation axis for $n$-body $RE$ on $\mathbb{R}^2$
is one of the eigenvectors (principal axis) of the inertia tensor
$I$.
\end{lemma}

We take the $z$ direction as the direction of $\Omega$,
$\Omega=\omega e_z$, where $\omega>0$ and $e_z$ is the
unit vector on the $z$ direction.
The other mutually orthogonal directions to $e_z$ are $e_x$, $e_y$.
Then, we get

\begin{equation}
Ie_z=\lambda e_z,
\end{equation}
with the eigenvalue
\begin{equation}
\lambda=e_z^T I e_z=\sum_\ell m_\ell (1-(q_\ell\cdot e_z)^2)
=\sum_\ell m_\ell \sin^2(\theta_\ell),
\end{equation}
and
\begin{equation}\label{defIxzIyz}
(I_{xz}, I_{yz})=(e_x^T I e_z, e_y^T I e_z)
=-\sum_\ell m_\ell \sin\theta_\ell \cos\theta_\ell (\cos\phi_\ell,\sin\phi_\ell)
=0.
\end{equation}
In the right hand side, we have used spherical coordinates,
and the fact that, $\cos\theta_k=(q_k\cdot e_z)$.

\begin{lemma}\label{IxzIyzEq0}
If $e_z$ is one of the eigenvectors of the inertia tensor,
then $I_{xz}=I_{yz}=0$.
The inverse is also true.
\end{lemma}
\begin{proof}
It is obvious, because $I e_z=I_{xz}e_x+I_{yz}e_y+I_{zz}e_z$.
\end{proof}

The following alternative expression for $I_{xz}=I_{yz}=0$ is useful
to study $RE$,
\begin{equation}\label{alternative}
\sum_{i=1,2,\dots n} 
m_i \sin(\theta_i)\cos(\theta_i) \left(\cos(\phi_k - \phi_i),\sin(\phi_k - \phi_i)\right) = 0,
\end{equation}
for each $k=1,2,\dots n$,
because all the terms in the right hand side are constant in time for $RE$.

\begin{remark}
The inequality 
$\lambda=\sum_\ell m_\ell \sin^2\theta_\ell \le M$
is satisfied. The equality holds if and only if
$\theta_\ell=\pi/2$ for all $\ell$,
namely, if all the masses are on the equator.
Therefore, the eigenvalue for non-collinear distribution of mass
satisfies $\lambda < M$.
\end{remark}

\begin{remark}\label{remarkForThreeEigenvectorsI}
Since the inertia tensor $I$ is an orthogonal 3 by 3 matrix,
there are three eigenvectors.
The above Lemma and the Remark hold for all three eigenvectors.
Let them be $e_\alpha$ $(|e_\alpha|=1)$ with corresponding eigenvalue $\lambda_\alpha$
for $\alpha=1,2,3$.
Then $I \tilde e=\tilde e\Lambda$,
where $\tilde e=(e_1,e_2,e_3)$ 
is the 3 by 3 orthogonal matrix
aligning the three eigenvectors $e_\alpha$,
and $\Lambda$ is the diagonal matrix of the eigenvalues
$\Lambda_{\alpha\alpha}=\lambda_\alpha$.

Let be $(q_\ell \cdot e_\alpha)=\cos\theta_{\ell\alpha}$,
$\ell=1,2,\dots,n$ and $\alpha=1,2,3$.
Then by the definition,
$(\cos\theta_{\ell 1},\cos\theta_{\ell 2},\cos\theta_{\ell 3})$
are the direction cosines for the position of the mass $\ell$ with respect to the orthogonal coordinates $e_1,e_2,e_3$,
and $\sum_\alpha \cos^2(\theta_{\ell\alpha})=1$.

The inequality $\lambda_\alpha<M$ is satisfied for all 
$\alpha=1,2,3$ for non-collinear distribution of masses.
\end{remark}

For collinear $RE$,
the definition of $I$ and the fact that $I_{xz}=I_{yz}=0$
is enough to investigate $RE$.
However, to investigate non-collinear $RE$, 
a deeper understanding of what $I_{xz}=I_{yz}=0$  
is required. We will do it in the next subsection.

\subsection{For the three-body problem on $\mathbb{S}^2$}\label{prelimsEqOfMotion}
In this subsection,
we consider the three body problem.
We don't fix the $z$-axis first.
Instead, we consider a  method to find
the eigenvectors of the inertia tensor
for given masses $m_k$ 
and arc angles $\sigma_{ij}$.
We don't restrict ourself to $RE$.
We consider general $m_k$ and $\sigma_{ij}$
with inequalities for the three sides 
$\sigma_{ij}$ to form a triangle,
$\sigma_{ij}\le \sigma_{jk}+\sigma_{ki}$
for $(i,j,k)=(1,2,3),(2,3,1),(3,1,2)$,
and $\sigma_{12}+\sigma_{23}+\sigma_{31}\le 2\pi$.

From now on, 
the  cyclic expression of 
$(i,j,k)=(1,2,3),(2,3,1),(3,1,2)$
will be denoted by
$(i,j,k) \in cr(1,2,3)$.
Also, we simply write $\{m_k,\cos\sigma_{ij}\}$
for the set 
$\{m_1,m_2,m_3,\cos\sigma_{12},\cos\sigma_{23},\cos\sigma_{31}\}$.

We found that the following 
matrix $J$ is similar to $I$ and is very useful,
\begin{equation}\label{defJ}
J=\left(\begin{array}{ccc}
m_2+m_3 & -\sqrt{m_1m_2}\cos\sigma_{12} & -\sqrt{m_1m_3}\cos\sigma_{13} \\
-\sqrt{m_2m_1}\cos\sigma_{21} &m_3+m_1& -\sqrt{m_2m_3}\cos\sigma_{23} \\
-\sqrt{m_3m_1}\cos\sigma_{31} & -\sqrt{m_3m_2}\cos\sigma_{32} & m_1+m_2
\end{array}\right).
\end{equation}

Or, equivalently,
\begin{equation}\label{defIByVectors}
J_{ij}=M\delta_{ij}-\sqrt{m_im_j}
	(q_i\cdot q_j),
\end{equation}
where $\delta_{ij}$ is the Kronecker delta.
The similarity will be proved ahead in the paper.
Before that, we investigate the properties of $J$.

\begin{theorem}\label{threeEquivalentStatements}
Assume that the $z$--axis is one of the eigenvectors of $I$,
we define $\cos\theta_k=(q_k\cdot e_z)$ as above.
Then,
\begin{equation}\label{defPsi}
\Psi
=(\sqrt{m_1}\cos\theta_1,\sqrt{m_2}\cos\theta_2,\sqrt{m_3}\cos\theta_3)^T/\sqrt{\sum\nolimits_{\ell=1,2,3}m_\ell \cos^2(\theta_\ell)}
\end{equation}
is the eigenvector of $J$
that 
belongs to the eigenvalue
$\lambda=\sum_{\ell=1,2,3}m_\ell \sin^2(\theta_\ell)$.
\end{theorem}

\begin{proof}
Neglecting the normalization factor, 
let us prove $Jv=\lambda v$ for $v_i=\sqrt{m_i}(q_i\cdot e_z)$.
Just by the definitions, we obtain
\begin{equation}\label{Jv1}
\begin{split}
(Jv)_i
&=\sum\nolimits_\ell (M\delta_{i\ell}-\sqrt{m_i m_\ell}(q_i\cdot q_\ell))
	\sqrt{m_\ell}(q_\ell \cdot e_z)\\
&=\sqrt{m_i} q_i^T\left(ME - \sum\nolimits_\ell m_\ell q_\ell q_\ell^T\right)e_z\\
&=\sqrt{m_i} q_i^T I e_z.
\end{split}
\end{equation}
By the assumption of this theorem, $Ie_z=\lambda e_z$,
the right hand side is
$\lambda v_i$.
\end{proof}

Note that the eigenvalue $\lambda$ is the same as for $I$.
So, we have the following lemma.
\begin{lemma}\label{ForThreeEigenvectorsII}
The inertia tensor $I$ and the matrix $J$ are similar.
\end{lemma}
\begin{proof}
There are three eigenvectors and eigenvalues for $I$.
Let them be $e_\alpha$ and $\lambda_\alpha$, $\alpha=1,2,3$.
Namely, $I e_\alpha=\lambda_\alpha e_\alpha$.
Theorem \ref{threeEquivalentStatements} shows that
there are corresponding eigenvector $\Psi_\alpha$ for $J$
with the same $\lambda_\alpha$.

If the eigenvalues are not degenerated,
the orthogonality $(\Psi_\alpha\cdot\Psi_\beta)=\delta_{\alpha\beta}$
is obvious.
But the orthogonality is true even for the degenerated case.

For $i=1,2,3$, we denote to the component $i$ of $\Psi_\alpha$ as 
$\Psi_{i\alpha}$ which is defined by \eqref{defPsi}.
Namely $\Psi_{i\alpha}=\sqrt{m_i}(q_i\cdot e_\alpha)/n_\alpha$,
$n_\alpha=\sqrt{\sum_\ell m_\ell (q_\ell\cdot e_\alpha)^2}$.
Then 
\[
\begin{split}
(\Psi_\alpha\cdot\Psi_\beta)
&=(n_\alpha n_\beta)^{-1}
\sum\nolimits_\ell m_\ell (q_\ell\cdot e_\alpha)(q_\ell\cdot e_\beta)\\
&=(n_\alpha n_\beta)^{-1}e_\alpha^T(M-I)e_\beta\\
&=(n_\alpha n_\beta)^{-1}(M-\lambda_\beta)\delta_{\alpha\beta}
=\delta_{\alpha\beta},
\end{split}
\]
because $e_\beta$ is the eigenvector of $I$ with the eigenvalue 
$\lambda_\beta$.

Therefore, 
the matrix $\widetilde\Psi=(\Psi_1,\Psi_2,\Psi_3)$
that is the alignment of the three column vectors $\Psi_\alpha$
is an orthogonal matrix, with $J\widetilde\Psi=\widetilde\Psi\Lambda$.
On the other hand, as stated in the Remark \ref{remarkForThreeEigenvectorsI},
$I\tilde e=\tilde e \Lambda$,
with the same $\Lambda$.
Therefore, $I$ and $J$ are similar.
Namely, $J=\widetilde\Psi\tilde e^T I e\widetilde\Psi^T$.
\end{proof}

\begin{remark}\label{correspondenceEamdPsi}
The correspondence between 
$e_\alpha$ and $\Psi_\alpha$ are now obvious. 
We denote as $b_\alpha$ the basis vectors  
$b_1=(1,0,0)^T$, $b_2=(0,1,0)^T$, $b_3=(0,0,1)^T$.
Then $e_\alpha=\tilde e b_\alpha$ and 
$\Psi_\alpha=\widetilde \Psi b_\alpha$.
\end{remark}

\begin{lemma}\label{cosFromPsi}
Let $\lambda$ and $\Psi$ be one of the eigenvalues and the corresponding eigenvector of $J$.
Then, $\cos\theta_k=(a_k\cdot e_z)$ for the corresponding $z$-axis
are uniquely determined  by
\begin{equation}\label{cosTheta}
(\sqrt{m_1}\cos\theta_1,
\sqrt{m_2}\cos\theta_2,
\sqrt{m_3}\cos\theta_3)^T
=\sqrt{M-\lambda}\,\,\Psi.
\end{equation}
\end{lemma}
\begin{proof}
Since
$M-\lambda=\sum_\ell m_\ell \cos^2(\theta_\ell)\ge 0$,
this lemma follows from the definition of $\Psi$ in \eqref{defPsi}.
\end{proof}

The inverse of Theorem \ref{threeEquivalentStatements} is now obvious, but it is well worth stating it here.

\begin{theorem}\label{PsitoEz}
If $\Psi$ is an eigenvector of $J$,
then the corresponding $e_z$ is the eigenvector of $I$.
with the same eigenvalue $\lambda_\alpha$.
\end{theorem}

\begin{proof}
It follows immediately  by  Lemma \ref{ForThreeEigenvectorsII}
and Remark \ref{correspondenceEamdPsi}.
\end{proof}

\begin{remark}
Theorem \ref{PsitoEz} is useful to find $RE$.
Here, we explain how to use it
For given $\{m_k,\cos\sigma_{ij}\}$,
let $\Psi$ be the eigenvector of $J$ with
eigenvalue $\lambda$.
We know $M\ge \lambda$. 
Then, by \eqref{cosTheta}, 
$\lambda=\sum_\ell m_\ell \sin^2(\theta_\ell)$.
By Theorem \ref{PsitoEz}, the corresponding $e_z$ is the eigenvector of $I$ with the same eigenvalue.
Then, by Theorem \ref{threeEquivalentStatements},
$\cos\theta_k$ are identified with $q_k\cdot e_z$.
Thus, the polar angles $\theta_k$ for the mass $k$ are identified.
\end{remark}

Before closing this section,
we prove the following Lemma.
\begin{lemma}\label{identity}
For the column vector
$v=(\sqrt{m_1}\cos\theta_1,\sqrt{m_2}\cos\theta_2,
\sqrt{m_3}\cos\theta_3)^T$,
\begin{equation}\label{identityI}
v^TJv=
\left(\sum_{\ell=1,2,3}m_\ell\cos^2(\theta_\ell)\right)
\left(\sum_{\ell=1,2,3}m_\ell\sin^2(\theta_\ell)\right)
-(I_{xz}^2+I_{yz}^2)
\end{equation}
is an identity.
\end{lemma}

\begin{proof}
Using equation \eqref{Jv1}, we get
\begin{equation}
\begin{split}
\sum_i v_i (Jv)_i
&=\sum_i m_i e_z^Tq_i q_i^T I e_z
=e_z^T(M-I)I e_z\\
&=MI_{zz}-e_zI^2e_z=MI_{zz}-(I_{xz}^2+I_{yz}^2+I_{zz}^2)\\
&=(M-I_{zz})I_{zz}-(I_{xz}^2+I_{yz}^2).
\end{split}
\end{equation}
Since $I_{zz}=\sum_\ell m_\ell \sin^2(\theta_\ell)$,
the Lemma is proved.
\end{proof}
This identity gives an alternative and direct
proof of Theorem \ref{PsitoEz}.
Because if $J\Psi=\lambda\Psi$
then $\Psi^T J \Psi=\lambda$,
which means $I_{xz}=I_{yz}=0$ by the identity.
Namely $e_z$ is the eigenvector of $I$.

\medskip

In this section, we have shown that
$I$ and $J$ are similar,
and that if we take one of the eigenvectors of $I$ as the $z$--axis, then corresponding $\Psi$ is an eigenvector of $J$.
The inverse is also true.

The eigenvalue equation $J\Psi=\lambda\Psi$
gives the relation between the set $\{m_k,\cos\sigma_{ij}\}$ and 
the set $\{\cos\theta_k\}$.

In the next sections,
we will get another relation for the same sets 
$\{m_k,\cos\sigma_{ij}\}$ and $\{\cos\theta_k\}$
from ``the equation of motion''.


\section{Equations of motion}\label{secEqOfMotion}
In this section 
we consider the equations of motion taking $e_z$ as the rotation axis
of a $RE$.
For the derivative of the potential $U$,
we use the notation
$U'(\cos\sigma_{ij})=dU(\cos\sigma_{ij})/d(\cos\sigma_{ij})$.
We  assume that $U'(\cos\sigma_{ij})$ is continuous and has definite sign for all ranges of $\sigma_{ij} \in (0,\pi)$. 
The sign $U'(\cos \sigma_{ij})>0$ stands for attractive force, and $<0$ for repulsive force.

The equations of motion are derived from the above Lagrangian through the Euler-Lagrange equations.
\begin{equation}
\begin{split}
&\frac{d}{dt}\left(m_i \dot\theta_i\right)
=m_i \sin\theta_i\cos\theta_i \dot\phi_i^2
	+\frac{\partial V}{\partial \theta_i},\\
&\frac{d}{dt}\left(m_i\sin^2(\theta_i)\dot\phi_i\right)
	=\frac{\partial V}{\partial \phi_i}.
\end{split}
\end{equation}
Since
$\dot \theta_k=0$ and 
$\dot\phi_k=\omega=$ constant for $RE$,
the left hand side of  both equations vanish.
Thus, the equations of motion for $RE$ are reduced to
\begin{equation}\label{eqOfMotForTheta}
\begin{split}
&\omega^2 m_i \sin\theta_i\cos\theta_i\\
&=\sum_{j\ne i}m_i m_jU'(\cos\sigma_{ij})
	\Big(
		\sin\theta_i\cos\theta_j-\cos\theta_i\sin\theta_j\cos(\phi_i-\phi_j)
	\Big).
\end{split}
\end{equation}
and
\begin{equation}\label{eqforphi}
\begin{split}
&m_1 m_2 U'(\cos \sigma_{12})\sin\theta_1 \sin\theta_2 \sin(\phi_1-\phi_2)\\
=&m_2 m_3 U'(\cos \sigma_{23})\sin\theta_2 \sin\theta_3 \sin(\phi_2-\phi_3)\\
=&m_3 m_1 U'(\cos \sigma_{31})\sin\theta_3 \sin\theta_1 \sin(\phi_3-\phi_1).
\end{split}
\end{equation}
Although the equations (\ref{eqOfMotForTheta}) and  (\ref{eqforphi}) do not involve time derivatives, they are derived immediately from the equations of motion and are the most important equations for $RE$, so they will be referred as 
``the equations of motion'' in the following.

\section{Euler relative equilibrium}\label{sec:Euler}
Most of the results described in this section can be generalized to the  
collinear $n$--body problem. In order to clarify the proofs we will restrict our analysis to the case $n=3$. 

\subsection{Geometry for the Euler relative equilibrium ($ERE$)}
We start with the following proposition.

\begin{proposition}
Taking the rotation axis as the $z$--axis, there are just two kinds of collinear $RE.$ All bodies are on the equator or they are on a rotating meridian.
\end{proposition}

\begin{proof}
When all bodies are on one geodesic,
they are on one plane that passes through the center of $\mathbb{S}^2$.
It is obvious that the normal vector to the plane
is an eigenvector of $I$ by its definition.
Taking this vector as the $z$-axis,
all bodies are on the equator.

The other two eigenvectors 
must be orthogonal to the normal vector.
Thus they must be on a plane, and
then all bodies are on one meridian.
\end{proof}

The above result was first proved just for the cotangent potential in \cite{zhu2}.

For the case when the bodies are on the equator, $\theta_k = \pi/2$ for all $k$, 
and then the relations between the shape variables $\sigma_{ij}$ 
and the configuration variables $\phi_k$ are trivial.
For this reason, our method does not contribute anything new.
We can find the $RE$ by
using elementary trigonometric elements,
see for instance \cite{Diacu1, zhu2}. 
We omit these kind of $RE$ in this paper (see for instance \cite{M-S, EPC2} where for the cotangent potential, even the stability of these $RE$ are studied).

For the case when the bodies are on a rotating meridian,
it is convenient to enlarge the range of $\theta_k$
to $-\pi\le \theta_k\le \pi$ with $\phi_k = 0$. 
The condition $I_{xz}=I_{yz}=0$ is reduced to
$\sum_\ell m_\ell\sin(2\theta_\ell)=0$
which is equivalent to diagonalize the two by two matrix
\begin{equation}
I_2=
\left(\begin{array}{ccc}
\sum_\ell m_\ell \cos^2(\theta_\ell) & -\sum_\ell m_\ell \sin\theta_\ell\cos\theta_\ell \\
-\sum_\ell m_\ell \sin\theta_\ell\cos\theta_\ell  & \sum_\ell m_\ell \sin^2(\theta_\ell)
 \end{array}\right).
\end{equation}

The characteristic polynomial for $I_2$ is
\begin{equation*}
p(\lambda)
=\det(\lambda-I_2)
=\lambda^2-M\lambda
+\sum\nolimits_{i<j}m_im_j\sin^2(\theta_{ij}),
\end{equation*}
where
\begin{equation}
\theta_{ij}=\theta_i-\theta_j.
\end{equation}

The discriminant $D$ for $p(\lambda)=0$ is
\begin{equation}
D=\sum\nolimits_\ell m_\ell^2
	+2\sum\nolimits_{i<j}m_im_j \cos(2\theta_{ij}).
\end{equation}

If $D\ne 0$, the matrix $I_2$ has two distinct eigenvectors.
Therefore, by  Lemma \ref{rotationAxis},
the $z$-axis is determined to be one of the two eigenvectors.
So, the angle $\theta_k$ can be determined.

\begin{lemma}\label{propTrans}
For a collinear $RE$ on a rotating meridian, if $D \neq 0$, then the formulae between the configuration variables $\theta_i$ and the shape variables 
$\theta_{ij}=\theta_i-\theta_j$ are given by 
\begin{equation}\label{collinear-condition}
\left( \cos (2\theta_i),\sin (2\theta_i) \right) = sA^{-1} \sum_{j=1,2,3} m_j \left( \cos (2\theta_{ij}),\sin (2\theta_{ij}) \right), 
\text{ for } i=1,2,3
\end{equation}
where $s = \pm 1$ and $A=\sqrt{D}$.
\end{lemma}
\begin{proof}
If $A\ne 0$, the equation 
$0=\sum_k m_k\sin(2\theta_k)
=\sum_k m_k \sin(2(\theta_1+\theta_{k1}))$
has two solutions
\begin{equation}\label{translationFormula}
\begin{split}
\cos(2\theta_1)&=s A^{-1}\left(
	m_1+m_2\cos\Big(2(\theta_1-\theta_2)\Big)
				+m_3\cos\Big(2(\theta_1-\theta_3)\Big)
	\right),\\
\sin(2\theta_1)&=s A^{-1}\left(
	m_2\sin\Big(2(\theta_1-\theta_2)\Big)
				+m_3\sin\Big(2(\theta_1-\theta_3)\Big)
	\right),
\end{split}
\end{equation}
where $s=\pm1$. The other angles $\theta_k$ are determined by 
$\theta_k=\theta_1+\theta_{k1}$.
\end{proof}

\begin{remark}
The two branches given for $s= \pm 1$ corresponds to the existence of  two  principal axes for $I_2$. 
We will show that ``the equations of motion'' determine the sign of $s$.
\end{remark}

\begin{lemma}\label{restri}
For a collinear $RE$ on a rotating meridian,
the shapes that give $D=0$ are restricted to the cases
$m_k<m_i+m_j$
and
\begin{equation}
\cos(2(\theta_i-\theta_j))=\frac{m_k^2-m_i^2-m_j^2}{2m_im_j}.
\end{equation}
for
$(i,j,k) \in cr(1,2,3)$.
\end{lemma}

\begin{proof}
Since
$D=\Big(m_1+m_2\cos(2\theta_{12})
				+m_3\cos(2\theta_{13})\Big)^2+
	\Big(m_2\sin(2\theta_{12})
				+m_3\sin(2\theta_{13})\Big)^2$,
$D=0$ yields
$m_1+m_2\cos(2\theta_{12})+m_3\cos(2\theta_{13})=0$
and $m_2\sin(2\theta_{12})+m_3\sin(2\theta_{13})=0$.
Equivalently,
$\sum_{\ell}m_\ell (\cos(2\theta_\ell),\sin(2\theta_\ell))=0$.
Then, Lemma  \ref{restri} follows.
\end{proof}

\begin{remark}
Lemma \ref{restri} shows that
the cases for $D=0$ are rare.
For given masses with $m_k<m_i+m_j$,
the number of solutions of $\theta_i-\theta_j$
are finite.
We can check
whether these shapes can form $RE$
with the given potential $V$ or not.
\end{remark}

\subsection{Equations of motion for $ERE$ on a rotating meridian}
Since in this case $\sin(\phi_i-\phi_j)=0$
and  $\cos(\phi_i-\phi_j)=1$ for all pair $(i,j)$,
``the equations of motion'' (\ref{eqforphi}) are satisfied,
and the equations (\ref{eqOfMotForTheta})
are reduced to
\begin{equation}\label{eqThetaForMeridian0}
\begin{split}
\frac{\omega^2}{2}m_k\sin(2\theta_k)
&= m_k \sum_{j \ne k} m_j\sin(\theta_k-\theta_j)U'( \cos (\theta_k - \theta_j)).\\
\end{split}
\end{equation}

\begin{proposition}\label{shapetoERE}
If $A\ne 0$, 
``the equations of motion'' for $ERE$ on a rotating meridian
are equivalent to the following equations,
\begin{equation}
\label{eqThetaForMeridian}
\begin{split}
&m_1m_2\left(s\frac{\omega^2}{2A}\sin\Big(2(\theta_1-\theta_2)\Big)
- \sin(\theta_1-\theta_2)U'(\cos (\theta_1 - \theta_2))
\right)\\
=&m_2m_3\left(s\frac{\omega^2}{2A}\sin\Big(2(\theta_2-\theta_3)\Big)
- \sin(\theta_2-\theta_3)U'(\cos (\theta_2 - \theta_3))
\right)\\
=&m_3m_1\left(s\frac{\omega^2}{2A}\sin\Big(2(\theta_3-\theta_1)\Big)
- \sin(\theta_3-\theta_1)U'(\cos (\theta_3 - \theta_1))
\right).
\end{split}
\end{equation}
\end{proposition}
\begin{proof} 
Using the  formulae \eqref{collinear-condition} between $\theta_k$ and $\theta_i - \theta_j$ given in Lemma \ref{propTrans} when $A\neq 0$ we obtain the result.
\end{proof}
For the same reason that we explained before, even when we do not have the time derivative in the above equation, 
the equation (\ref{eqThetaForMeridian})
will be referred to as ``the equation of motion''
for a $ERE$ on a rotating meridian.

Now we define the useful expressions 
\begin{equation}
\label{defOfFandG}
F_{ij} = m_im_j\sin(\theta_i-\theta_j)U'(\cos \theta_{ij}),\qquad 
G_{ij} = m_im_j\sin (2(\theta_i-\theta_j)).
\end{equation}
Then, the equations
\eqref{eqThetaForMeridian} can be written in a compact form as
\begin{equation}\label{eqcompact}
s\frac{\omega^2}{2A}(G_{ij}-G_{jk}) = F_{ij} - F_{jk} \quad
\mbox{ for }(i,j,k)\in cr(1,2,3).
\end{equation}

\begin{theorem}
[Condition for a shape]
\label{propConditionForShape}
If $A\ne0$,
\begin{equation}
det
=\left|\begin{array}{cc}
G_{12}-G_{23} & G_{31}-G_{12} \\
F_{12}-F_{23} & F_{31}-F_{12}
\end{array}\right|
=0
\end{equation}
is a necessary and sufficient condition for a shape to satisfy 
``the equations of motion'' \eqref{eqThetaForMeridian}, and then to generate an $ERE$.
\end{theorem}

\begin{proof}
If equations \eqref{eqcompact} are satisfied, then
$det=0$ is obvious.

Inversely, if $det=0$, then there are two cases.
The first case is when all elements of the matrix are zero.
For this case, ``the equations of motion'' 
\eqref{eqcompact}  are trivially satisfied and $\omega$ is undetermined.

The second case 
is when at least one of the elements of the matrix 
is not zero.
For example, let  $G_{12}-G_{23}\ne 0$.
We define
$s\omega^2$ by
$s\omega^2/(2A)=(F_{12}-F_{23})/(G_{12}-G_{23})$.
Then,
$det=0$ yields $s\omega^2/(2A)(G_{31}-G_{12})=F_{31}-F_{12}$.
Thus the equations  \eqref{eqcompact}
are satisfied.
Similarly, all other cases satisfy the equations
\eqref{eqcompact}.
\end{proof}

\begin{remark}
As shown in the previous example,
the sign of $(F_{12}-F_{23})/(G_{12}-G_{23})$ determines $s$.
Namely, ``the equations of motion'' determine $s$.
If this value is zero, then $\omega=0$ and the $ERE$ is a fixed point.
\end{remark}

We finish this subsection by setting
the following two corollaries.

\begin{cor}[Equilateral $RE$ on a rotating meridian]\label{equilateralRigidRotator}
For any three masses, the equilateral triangle
$\theta_1-\theta_2=\theta_2-\theta_3=\theta_3-\theta_1
=2\pi/3$ generates a $RE$ on a rotating meridian.
\end{cor}

\begin{proof}
For this shape, $A^2 = \sum_{i<j} (m_i-m_j)^2/2$.
If $m_1=m_2=m_3$ then $A=0$, 
the equations of motion \eqref{eqThetaForMeridian0} are satisfied with $\omega = 0$, the $RE$ is a fixed point. 
For the others cases $A  \neq 0$.  
Then the expressions in the three parentheses of equation \eqref{eqThetaForMeridian} are zero by taking 
$s\omega^2/(2A) = - U'(-1/2)$. 
An alternative proof for the non-equal masses case is obtained by using \eqref{collinear-condition}. We get $\sin (2\theta_k) = -s\sqrt{3}(m_i-m_j)/(2A)$, and we can verify easily that they satisfy 
equation \eqref{eqThetaForMeridian0}.
\end{proof}

\begin{cor}
If a shape on a rotating meridian generates a $RE$ for an
attractive potential $U$, then
the same shape generates a $RE$ for the repulsive potential $-U$.
The corresponding two configurations
differ only by an overall angle of $90$ degrees,
if not, it is a fixed point.
\end{cor}
\begin{proof}
If we change the potential $U$ (attractive) to $-U$ (repulsive),
the $F_{ij}-F_{jk}$ change their  sign,
and the equalities are satisfied by simply changing 
 $s \to -s$ which produces  $\theta_k \to \theta_k+\pi/2$.
\end{proof}


\section{Lagrange relative equilibrium}\label{sec:Lagrange}
In this section, we will show that the ``equations of motion'',
\eqref{eqOfMotForTheta} and \eqref{eqforphi},
require
a specific combinations of $\{m_k, \cos\sigma_{ij}\}$ 
and a specific  principal axis for rotation, in order to get a
 $LRE$.
We restrict our analysis to the case $n=3$.

\subsection{Equations of motion for $LRE$}
We start to list up
the necessary conditions
for the $LRE$.
\begin{proposition}\label{propLRE}
The necessary conditions to have a $LRE$ are
\begin{equation}
\label{conditionForLagrangian1}
\sin\theta_k\ne 0
\mbox{ and }
\sin(\phi_i-\phi_j)\ne 0
\mbox{ for all }k
\mbox{ and all pair }i,j, 
\end{equation}
\begin{equation}
\label{conditionForLagrangian2}
U'(\cos \sigma_{12})\cos\theta_3
=U'(\cos \sigma_{23})\cos\theta_1
=U'(\cos \sigma_{31})\cos\theta_2 \ne 0.
\end{equation}
\end{proposition}

\begin{proof}
From \eqref{eqforphi}, 
since we are assuming 
$U'(\cos \sigma_{ij})$ has definite sign for all $i,j$, then if
$\sin\theta_i\sin\theta_j\sin(\phi_i-\phi_j)=0$ for one pair $(i,j)$,
then the same happen for the others $(i,j)$
which implies that the three bodies must be on the same meridian.
Therefore, by the definition of $LRE$, we obtain \eqref{conditionForLagrangian1}.
Now, by the definition of spherical coordinates $\sin\theta_k>0$,
and using again that $U'(\cos\sigma_{ij})$
has definite sign, we obtain that
all $\sin(\phi_i-\phi_j)$ must have the same sign.
If one configuration with
\begin{equation}\label{signOfDeltaPhi}
\sin(\phi_i-\phi_j)
< 0 \quad 
\mbox{ for } \quad (i,j) \in (1,2), (2,3), (3,1),
\end{equation}
satisfies \eqref{eqOfMotForTheta} and \eqref{eqforphi},
then the configuration that has opposite orientation, 
that is $(\phi_i-\phi_j)\to -(\phi_i-\phi_j)$ and 
$\sin(\phi_i-\phi_j)>0$
also satisfies \eqref{eqOfMotForTheta} and \eqref{eqforphi}.
For concreteness, we assume \eqref{signOfDeltaPhi} in the following.

Multiplying by $\cos(\theta_3)$ both sides of
the first and the second expressions in equation (\ref{eqforphi}),
we obtain
\begin{equation*}
\begin{split}
&m_1m_2U'(\cos \sigma_{12})\sin\theta_1\sin\theta_2\sin(\phi_1-\phi_2)\cos\theta_3\\
&=m_2U'(\cos \sigma_{23})\sin\theta_2
(m_3\sin\theta_3\cos\theta_3\sin(\phi_2-\phi_3))\\
&=m_2U'(\cos \sigma_{23})\sin\theta_2
	(m_1\sin\theta_1\cos\theta_1\sin(\phi_1-\phi_2)).
\end{split}
\end{equation*}
To get the last line, we have used (\ref{alternative}) for $k=2$.
Dividing the first and the last line by
$m_1m_2\sin\theta_1\sin\theta_2\sin(\phi_1-\phi_2)\ne 0$,
we obtain $U'(\cos \sigma_{12})\cos\theta_3=U'(\cos \sigma_{23})\cos\theta_1$.
Similarly, we  obtain, the last part of equation  \eqref{conditionForLagrangian2}.

Again, if one $\cos\theta_k=0$, then all $\cos\theta_k$ must be zero since $U'(\cos \sigma_{ij})\ne 0$ and then 
the configuration is a collinear configuration on the Equator.
Therefore, for the $LRE$ we must have
$\cos\theta_k\ne 0$.
Then by using similar arguments as above, the common value in \eqref{conditionForLagrangian2} is not zero. This finishes the proof of Proposition \ref{propLRE}.
\end{proof}

For a $LRE$, we define the rotation axis as the $z$--axis.
Therefore, the equator is also defined.
Then, we obtain the following result.
\begin{cor}
\label{cosMustHaveTheSameSign}
The three bodies for $LRE$ are all on the northern hemisphere
or all on the southern hemisphere.
\end{cor}

\begin{proof}
Since we are assuming 
that $U'(\cos\sigma_{ij})$
has the same sign for all $(i,j)$, we obtain that
the three quantities $\cos\theta_k$ must have the same sign
by (\ref{conditionForLagrangian2}).
\end{proof}

Again, if a configuration with
\begin{equation}\label{positivecosine}
\cos \theta_k > 0 \quad \text{for all} \quad k=1,2,3,
\end{equation}
satisfies the equations of motion \eqref{eqforphi} and \eqref{eqOfMotForTheta}, then
the configuration with $\theta_k \to \pi-\theta_k$ that has $\cos\theta_k<0$
also satisfies the same equations.
For concreteness, we assume \eqref{positivecosine} in the following.

\begin{remark}
The choice $\cos\theta_k>0$ for all $k=1,2,3$ joint with the fact that
$\sin\phi_{ij}$ has the same sign for 
$(i,j) \in (1,2), (2,3),(3,1)$
means that the north pole is inside of the smaller region 
surrounded by the minor 
geodesics connecting the three bodies.
See Figures \ref{figisoscelesSigma12EqPiDiv3and2PiDiv3} and 
\ref{figIsoscelesSigma12EqPiDiv6}.
\end{remark}

\begin{remark}
Corresponding to the choice of the sign for 
$\cos\theta_k$ and $\sin(\phi_i-\phi_j)$,
there are four $LRE$ with the same 
arc angles $\sigma_{ij}$.
\end{remark}

Now we give an important necessary conditions
for $LRE$.

\begin{proposition}
The necessary conditions
to have a $LRE$ are 
\begin{equation}
\label{equivalentlagrange}
\frac{U'(\cos \sigma_{12})}{\cos\theta_1 \cos\theta_2} =
\frac{U'(\cos \sigma_{23})}{\cos\theta_2 \cos\theta_3} =
\frac{U'(\cos \sigma_{31})}{\cos\theta_3 \cos\theta_1} =
\frac{\omega^2}{\sum_k m_k \cos^2 \theta_k}.
\end{equation}
\end{proposition}

\begin{proof}
By the equation (\ref{conditionForLagrangian2}),
let be $U'(\cos \sigma_{ij}) = \gamma \cos\theta_i\cos\theta_j$
with common function $\gamma$.
Substituting this expression into ``the equations of motion'' \eqref{eqOfMotForTheta},
we obtain  the value 
$\gamma = \omega^2/(\sum_k m_k \cos^2 \theta_k)$.
\end{proof}

\begin{theorem}\label{theCoditionInSigmaAndTheta}
The necessary and sufficient conditions 
for three bodies to form a $LRE$
on $\mathbb{S}^2$ are
$\sin\theta_k\ne 0$, $\cos\theta_k\ne 0$,
and \eqref{equivalentlagrange},
where $\cos\theta_k=(q_k\cdot e_z)$
and $e_z$ is one of the eigenvectors of $I$.
\end{theorem}

\begin{proof}
The necessity was already shown.
Inversely,
put $U'(\cos\sigma_{ij})$ in \eqref{equivalentlagrange}
into the equations of motion \eqref{eqOfMotForTheta}
and \eqref{eqforphi}.
Then, using $I_{xz}=I_{yz}=0$ in the form given by equation \eqref{alternative},
we can verify that the equations
\eqref{eqOfMotForTheta} and \eqref{eqforphi}
are satisfied.
\end{proof}

We have seen 
that the rotation axis for the $LRE$ is one of the three principal axes of the matrix $I$, the next result 
tells us which one is exactly the rotation axis.

\begin{proposition}\label{prop-eigen}
Taking the $z$-axis as the rotation axis for a $LRE$,
the corresponding eigenvector $\Psi$ of $J$ is given by
\begin{equation}\label{the-eigenvector}
\Psi_L = \left( \frac{\sqrt{m_1}}{U'(\cos \sigma_{23})}, \frac{\sqrt{m_2}}{U'(\cos \sigma_{31})}, \frac{\sqrt{m_3}}{U'(\cos \sigma_{12})} \right)
/
\left( \sum m_k/(U'(\cos \sigma_{ij}))^2 \right)^{1/2},
\end{equation}
where the sum runs for $(i,j,k)\, \in \, cr(1,2,3).$
\end{proposition}

\begin{proof} 
By  equation 
\eqref{conditionForLagrangian2}, $\cos \theta_k = \delta  /U'(\cos \sigma_{ij})$,
with $\delta \ne 0$.
Thus, the normalized eigenvector 
$\Psi$
in Theorem \ref{threeEquivalentStatements} is given by the equation \eqref{the-eigenvector}.
\end{proof}

\begin{remark}
Note that $J$ and $\Psi_L$ are functions 
of the set $\{m_k, \cos\sigma_{ij}\}$.
Therefore, the eigenvalue problem $J\Psi_L=\lambda\Psi_L$
gives the explicit necessary condition for the set.
We will show that this is also the sufficient condition for the set.
\end{remark}

\begin{remark}
Even when $J$ is degenerated, ``the equations of motion'' determine
the rotation axis.
For example for the case $m_k=m$ and $\sigma_{ij}=\pi/2$,
$J$ is $2m$ times the identity matrix. 
Therefore, any direction is a principal axis of $J$.
But,  Proposition \ref{prop-eigen} states that
``the equations of motion''
determine $\Psi_L=(1,1,1)/\sqrt{3}$,
which corresponds to $\cos\theta_k=1/\sqrt{3}$.
\end{remark}

\begin{cor}
For a $LRE$ the angular velocity $\omega$ is given by  
\begin{equation}\label{omega}
\omega^2 = U'(\cos \sigma_{12})U'(\cos \sigma_{23})U'(\cos \sigma_{31}) \left( \sum m_k/(U'(\cos \sigma_{ij}))^2 \right). 
\end{equation}
\end{cor}
\begin{proof}
The equation \eqref{equivalentlagrange} and $\Psi=\Psi_L$ 
yields this expression.
\end{proof}

Now, we have a sufficient amount of information  to establish a sufficient condition for a LRE.

\begin{theorem}\label{ConditonForLRE}
The necessary and sufficient condition
for the set $\{m_k, \cos\sigma_{ij}\}$
to form a $LRE$ is that the matrix $J$ has the eigenvector $\Psi_L$ given by \eqref{the-eigenvector}.
\end{theorem}

\begin{proof}
The necessity was already shown in Proposition \ref{prop-eigen}. 

Inversely, if $\Psi_L$ in \eqref{the-eigenvector}
is an eigenvector of $J$,
then identify $\Psi_L=\Psi$,
by Theorem \ref{PsitoEz},
the corresponding $z$-axis is one of the eigenvectors of $I$.
The quantities $\cos \theta_k$ are uniquely determined by  \eqref{cosTheta} as
\begin{equation}\label{cosThetaII}
\cos\theta_k =  \sqrt{(M-\lambda)}/
\left(U'(\cos \sigma_{ij})
\sqrt{\sum_{i'j'k'} m_{k'}/(U'(\cos \sigma_{i'j'}))^2}
\right).
\end{equation}

Now, let $\omega^2$ be equal to the value in \eqref{omega}.
Then, we can easily verify  that equations \eqref{equivalentlagrange} are satisfied.
\end{proof}

\begin{remark}
As shown above, if $\{m_k, \cos\sigma_{ij}\}$ turns out to be able to form a $LRE$,
$\cos\theta_k$ and $\omega^2$ are uniquely determined by 
 $\sigma_{ij}$. Then, $\phi_i-\phi_j$ are also determined
by \eqref{fundamentalrelation} and \eqref{signOfDeltaPhi}.
Thus, one $LRE$ is determined.
Of course, there are another three $LRE$ that have the same 
$\cos\sigma_{ij}$
corresponding to the choice of the sign of $\sin(\phi_i-\phi_j)$ and $\cos\theta_k$.
\end{remark}

Before closing this section, we give two corollaries
and a trivial proposition.

\begin{cor}[Equilateral $LRE$ requires equal masses]\label{equilateralequal}
An equilateral triangle is a $LRE$ if and only if the masses are equal and the angles $\theta_k$ are equal
\cite{Diacu-EPC1}.
\end{cor}

\begin{proof}
Let $\sigma=\sigma_{12}=\sigma_{23}=\sigma_{31}$. By equation \eqref{conditionForLagrangian2} we have $\cos \theta_1 = \cos \theta_2 = \cos \theta_3$.
Therefore, the matrix $J$ has the eigenvector 
$(\sqrt{m_1},\sqrt{m_2},\sqrt{m_3})$, that is 
\begin{equation*}
\left(\begin{array}{ccc}
m_2+m_3 & -\sqrt{m_1m_2}\cos\sigma & -\sqrt{m_1m_3}\cos\sigma \\
-\sqrt{m_2m_1}\cos\sigma &m_3+m_1& -\sqrt{m_2m_3}\cos\sigma \\
-\sqrt{m_3m_1}\cos\sigma & -\sqrt{m_3m_2}\cos\sigma & m_1+m_2
\end{array}\right)
\!\!\!
\left(\begin{array}{c}
\sqrt{m_1}\\\sqrt{m_2}\\\sqrt{m_3}
\end{array}\right)
=\lambda
\!\!
\left(\begin{array}{c}
\sqrt{m_1}\\\sqrt{m_2}\\\sqrt{m_3}\end{array}\right).
\end{equation*}

This equation reduces to
\begin{equation*}
\lambda
=(m_1+m_2)(\cos\sigma-1)
=(m_2+m_3)(\cos\sigma-1)
=(m_3+m_1)(\cos\sigma-1),
\end{equation*}
which means that $m_1=m_2=m_3$ because $\cos\sigma\ne 1$. 
\end{proof}

\begin{cor}
There are no
 $LRE$ on $\mathbb{S}^2$ for any repulsive force.
\end{cor}
\begin{proof}
Since $U'(\cos \sigma_{ij}) < 0$ for any repulsive force, it can not satisfy 
equation \eqref{omega}.
\end{proof}

\begin{proposition}
There are no fixed point $LRE$.
\end{proposition}
\begin{proof}
By (\ref{omega}),
$\omega\ne 0$ for $U'(\cos\sigma_{ij})\ne 0$.
\end{proof}


\section{Equal masses $3$--body $RE$ on $\mathbb{S}^2$ under the
cotangent potential}\label{sec:curved}
The three body problem on $\mathbb{S}^2$ under the cotangent potential
is the natural extension of the planar Newtonian three body problem to the sphere \cite{Borisov2, Diacu1}.

The cotangent potential and its derivative are given by \cite{Diacu-EPC1}
\begin{equation*}
U(\cos\sigma_{ij})
=\frac{\cos\sigma_{ij}}{\sqrt{1-\cos^2(\sigma_{ij})}},
\qquad
U'(\cos\sigma_{ij})=\frac{1}{\sin^3(\sigma_{ij})}.
\end{equation*}

Since $U'(\cos\sigma_{ij})>0$, this potential produces the attractive force
among the masses.
In this section we only consider the equal masses case $m_k=m$.

The reason why we decided to include this problem in the paper is by one hand, to show how our method is easy to apply and allow us to find new families of relative equilibria. By the other hand, since this problem has been widely studied in the last years (see for instance \cite{Diacu1, Diacu3, Diacu-EPC1, tibboel} and the references therein), 
we can compare the results that we obtain  with the results obtained by other authors using different approaches. In a forthcoming 
paper we will present several new families of relative equilibria for the general problem, for now we only show briefly the equal masses case.

\subsection{Euler relative equilibria}
Equal masses $ERE$ are completely determined by our method.
We will show that there are scalene and isosceles $ERE$.

Since we extended the range of $\theta_k$ to $-\pi\le \theta_k \le \pi$, $U'$ should be written by $U'(\cos\theta_{ij})=1/|\sin^3\theta_{ij}|$.
We write $a=\theta_{21}$ and $x=\theta_{31}$.
Without loss of generality, we can restrict $0<a<\pi$ and $-\pi<x<\pi$.
The discriminant is $D=A^2=3+2(\cos(2a)+\cos(2x)+\cos(2(x-a)))$.

\subsubsection{The degenerated case}
The degenerated case $A=0$ is realized by 
$\cos(2a)=\cos(2x)=-1/2$ and $\sin(2a)+\sin(2x)=0$.
The solutions are the following.
\begin{enumerate}
\item The equilateral triangle: $a=2\pi/3$ and $x=-2\pi/3$.
\item Isosceles triangles with equal angles $\pi/3$. 
Namely, $(a,x)=(\pi/3,2\pi/3)$, $(\pi/3,-\pi/3)$, and $(2\pi/3,\pi/3)$.
\end{enumerate}

For the first shape, the equations of motion are
satisfied if and only if $\omega=0$.
Namely, this shape describes a fixed pint.

For the second shapes, the equations of motion determine
uniquely $\theta_k$ and $\omega$.
For example, for $(a,x)=(2\pi/3,\pi/3)$,
the equations of motion determine
$\theta_3=0$, $\theta_2=-\theta_1=\pi/3$, and $\omega^2=16/(3\sqrt{3})$.

\subsubsection{Non-degenerated case}
For the case $A\ne0$, the condition for the shape to form an
$ERE$ is
\begin{equation*}
\begin{split}
det=&\frac{g}{(\sin x |\sin x|)(\sin a |\sin a|)(\sin(x-a) |\sin(x-a)|)}=0, \quad \text{where}\\
g=
&\sin(x)|\sin(x)|\Big(\sin(2x)+\sin(2a)\Big)
	\Big(\sin(x-a)|\sin(x-a)|-\sin a |\sin a|\Big)\\
&-\sin(x-a)|\sin(x-a)|\Big(\sin(2a)-\sin(2(x-a))\Big)
	\Big(\sin a |\sin a|+\sin x |\sin x|\Big).
\end{split}
\end{equation*}

\begin{figure}
   \centering
   \includegraphics[width=8cm]{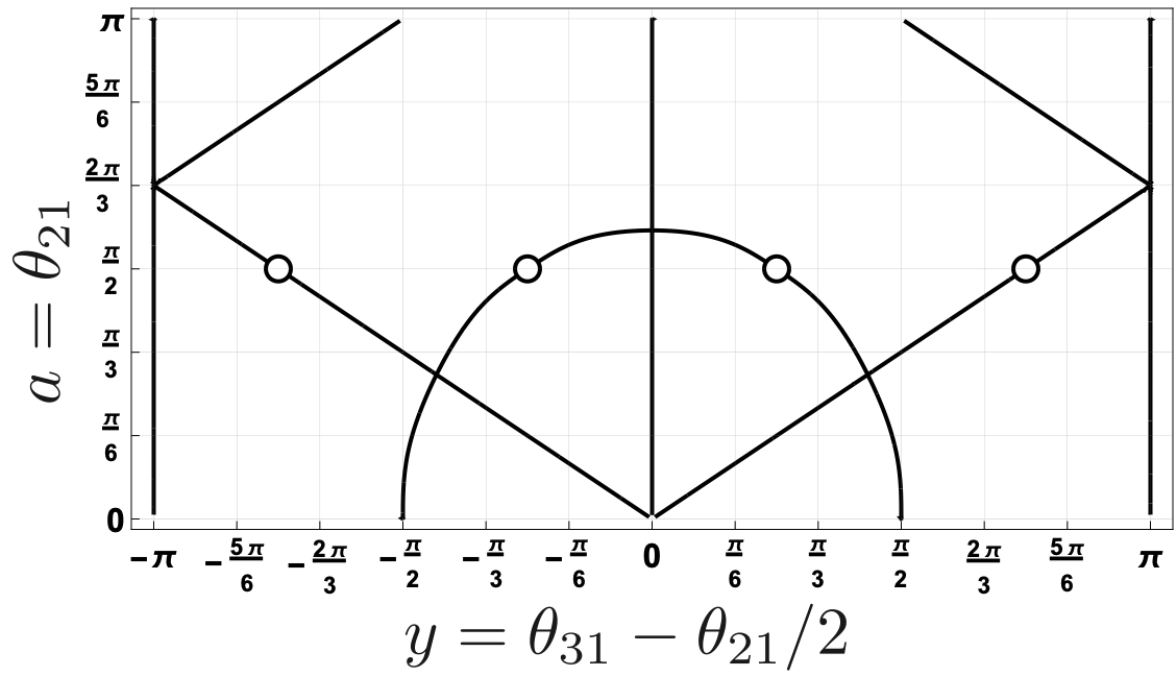}
   \caption{The contour for $det=0$.
   The vertical and horizontal axes represent $a=\theta_{21}$ and
   $y=\theta_{31}-\theta_{21}/2=x-a/2$ respectively.
   Two vertical lines $y=\pm \pi$ represent the same line.
   The curve and the straight lines represent
   scalene and isosceles triangle $ERE$ respectively.
   The four points 
   $(y,a)=(\pm\pi/4,\pi/2), (\pm 3\pi/4,\pi/2)$ are excluded.
   The point $(y,a)=(\pm \pi,2\pi/3)$ represents the equilateral triangle 
   which is the fixed point.}
   \label{figDet}
\end{figure}
\begin{figure}
   \centering
   \includegraphics[width=3.5cm]{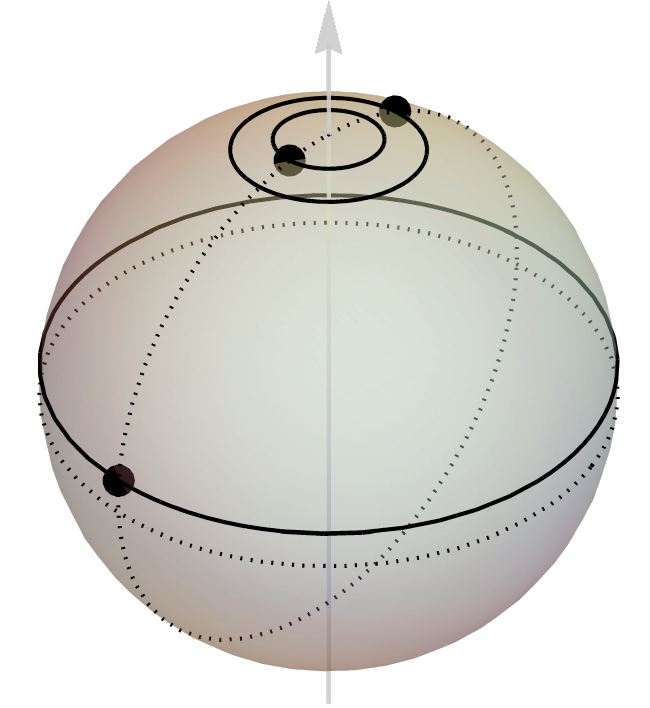}
   \includegraphics[width=3.5cm]{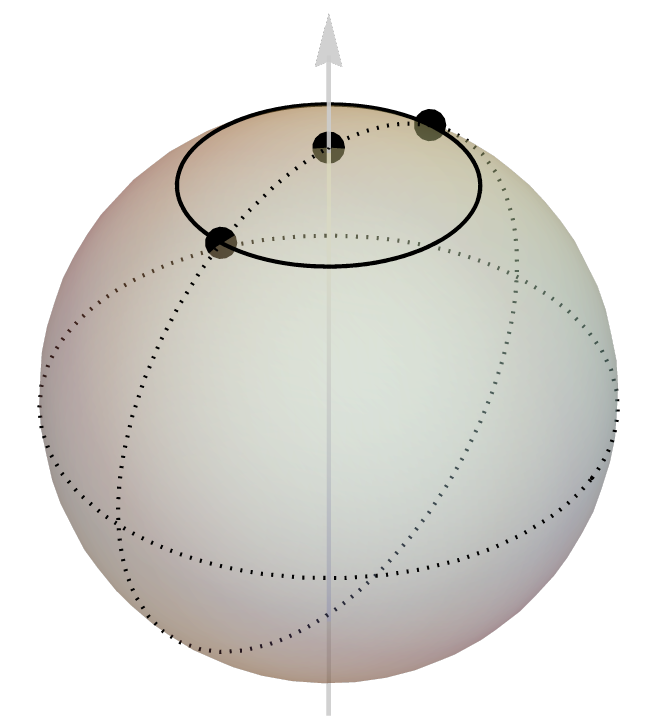}
   \includegraphics[width=3.5cm]{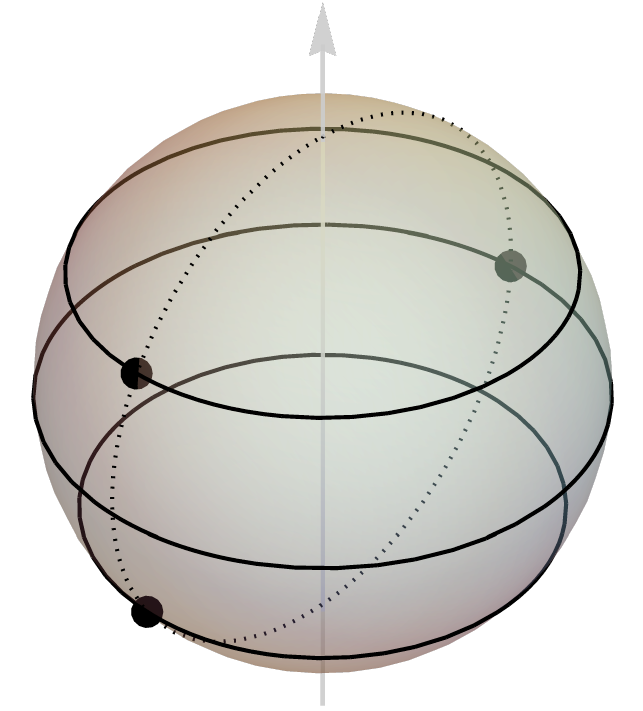}
   \caption{Three typical $ERE$ on a rotating meridean.
   Left: A scalene configuration, where three arc angles between bodies are different.
   Middle: An isosceles configuration with two equal arc angles $\theta$ are smaller than $2\pi/3$ and $\theta\ne \pi/2$, 
   in this case the middle mass must be at one of the poles.
   Right: An isosceles configuration with $2\pi/3<\theta < \pi$, 
   in this case the middle mass must rotates on the equator.}
   \label{Eulerian-scalene}
\end{figure}
The contour $det=0$ is shown in Figure \ref{figDet}.
The curve represents the scalene triangle $ERE$ where all $|\theta_{ij}|$ are different,
and the straight lines represent the isosceles triangle $ERE$.
The equations $det=0$ and $g=0$ are invariant for
the exchange $a \leftrightarrow x$,
or $a\to -a$ and $x\to x-a$.
These invariances are induced by the exchange of masses.

For the scalene triangle, the curve in the region $-a/2<y=x-a/2<a/2$ is represented by
\begin{equation}
\label{solOfCos2y}
\cos(2y)
=\cos(a)+\frac{\sin^2(a)}{\cos(a)}\left(
	\cos(2a)
	+\sqrt{\cos^2(2a)-4\cos(2a)-4}
	\right).
\end{equation}

Let $a_\ell$ be the largest angle of $|\theta_{ij}|$.
Numerical calculations suggest that
the scalene arc in the Figure \ref{Eulerian-scalene} is convex.
So, the curve for scalene triangle appears when $a_\ell$ satisfies
$\pi/2<a_\ell<a_c=1.8124...$,
where $a_c$ is the solution of equation \eqref{solOfCos2y} for $y=0$.
Namely,
\begin{equation*}
\cos(a_c)
=-1+2^{-1}\left(
				\left(1+9^{-1}\sqrt{78}\right)^{1/3}
				+
				\left(1-9^{-1}\sqrt{78}\right)^{1/3}
				\right).
\end{equation*}
Since $\pi/2<a_\ell$, 
the longest arc length $Ra_\ell$ goes to infinity as $R\to \infty$.
Therefore, the scalene triangle $ERE$ has no  Euclidean limit.

We can solve explicitly the scalene $ERE$ in $-a/2<y<a/2$ for given $\cos a$
using \eqref{solOfCos2y} and usual trigonometric computations.
The solutions for the other regions are given by the symmetry mentioned above.

For the isosceles triangles, we take $m_3$ at the mid point between $m_1$ and $m_2$ with $\theta=\theta_2-\theta_3=\theta_3-\theta_1$, $0<\theta<\pi$.
We can show that
\begin{itemize}
\item[i)] $\theta_3=0 \mod(\pi)$,
$\omega^2=+f(\theta)$ \quad for $\theta\in (0,2\pi/3) 
\setminus
\{\pi/2\}$,
\item[ii)]  $\theta_3$ is arbitrary and $\omega^2=0$ \quad for $\theta=2\pi/3$,
\item[iii)] $\theta_3=\pi/2 \mod(\pi)$,
$\omega^2=-f(\theta)$ \quad for 
$\theta\in (2\pi/3,\pi)$.
\end{itemize}
Where
$f(\theta)=2\left(1/|\sin(2\theta)|^3+1/\big(\sin^2(\theta)\sin(2\theta)\big)\right)$.
See Figure \ref{Eulerian-scalene}.

\subsection{Lagrange relative equilibria}

Finally we consider  $LRE$ for equal masses case.
The condition for the shape $\sigma_{ij}$ to form a $LRE$,
$J\Psi_L=\lambda\Psi_L$ is
\begin{equation}\label{eqForEqualMassLRE}
2-\lambda/m
=\frac{\cos(\sigma_{jk})\sin^3(\sigma_{ki})+\sin^3(\sigma_{jk})\cos(\sigma_{ki})}{\sin^3(\sigma_{ij})},
\end{equation}
for $(i,j,k)\in cr(1,2,3)$.

Our numerical calculations suggest that
there are no scalene triangles solutions for this equation.
However, further investigations are needed to prove this statement.

Let us proceed to the analysis of the isosceles triangles $LRE$.
Let be $\sigma_{23}=\sigma_{31}=\sigma$. 
Then, equation \eqref{eqForEqualMassLRE} is reduced to
\begin{equation}
q(\sigma,\sigma_{12})
=\cos\sigma
\big(2\sin^6(\sigma)-\sin^6(\sigma_{12})\big)
-\sin^3(\sigma)\cos(\sigma_{12})\sin^3(\sigma_{12})
=0.
\end{equation}
Obviously,
equilateral triangles $\sigma=\sigma_{12}$
satisfy $q=0$.
The graphical representation of this equation is
shown in Figure \ref{figMspecialCase8equalMassesCaseCont3FigSolution}.
\begin{figure}
   \centering
   \includegraphics[width=7cm]{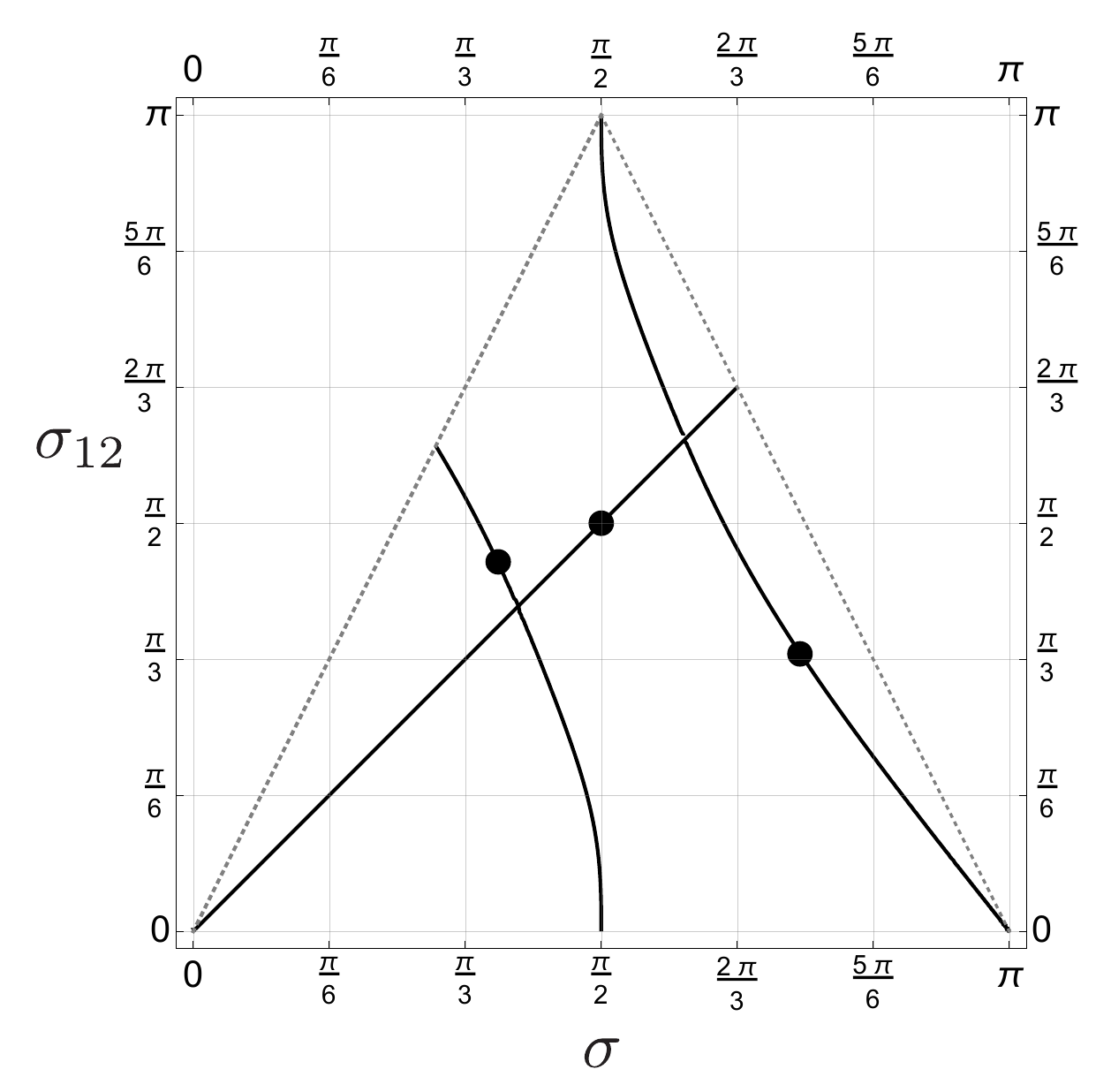} 
   \caption{The solid line represents the shape for the equal masses isosceles 
   $LRE$
   with $\sigma_{12}$ and $\sigma=\sigma_{23}=\sigma_{31}$.
   The straight line represents
   equilateral triangle.
   The region inside the 
  dotted lines is 
   the region to form a triangle, 
   $\sigma_{12}<2\sigma<2\pi-\sigma_{12}$.
   Note that the curve is point symmetric
   around $(\sigma,\sigma_{12})=(\pi/2,\pi/2)$.
   The three  black circles represent
   the right-angled triangles,
   where the angle at the vertex
   $m_3$ is $\pi/2$.
   }
   \label{figMspecialCase8equalMassesCaseCont3FigSolution}
\end{figure}

\begin{figure}
   \centering
\includegraphics[width=3.5cm]{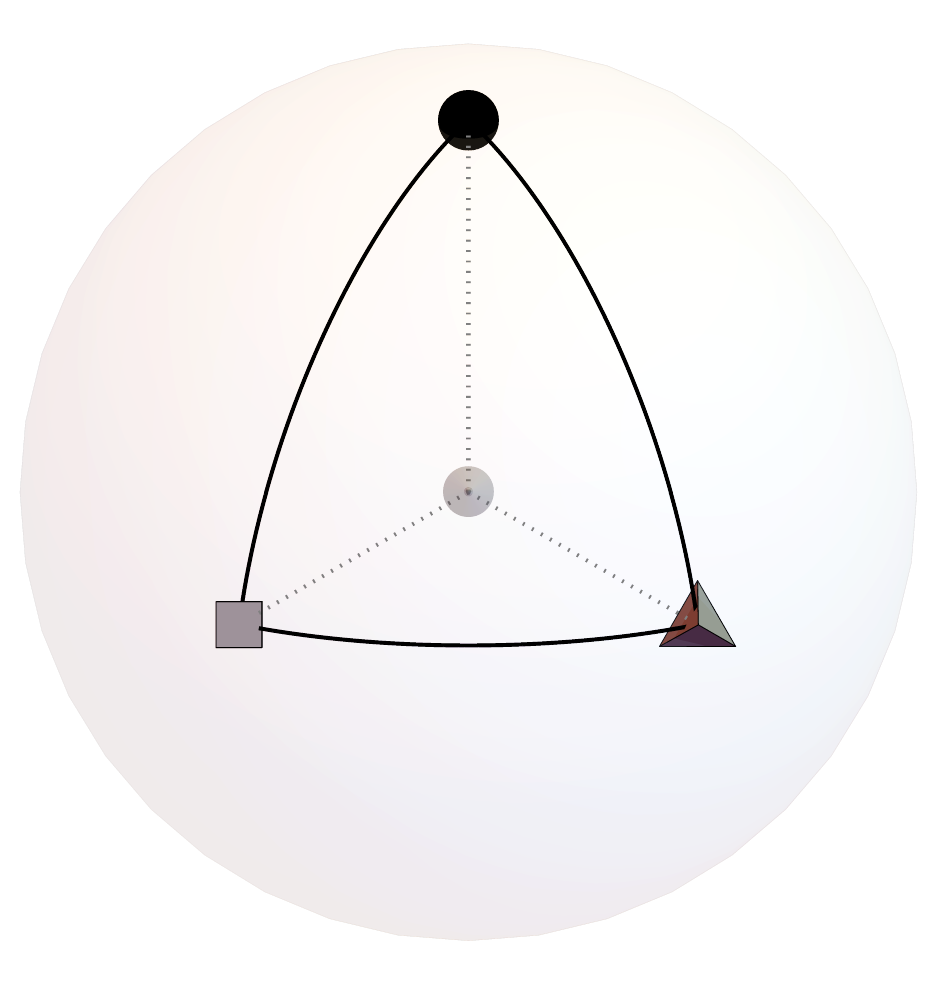}
\includegraphics[width=3.5cm]{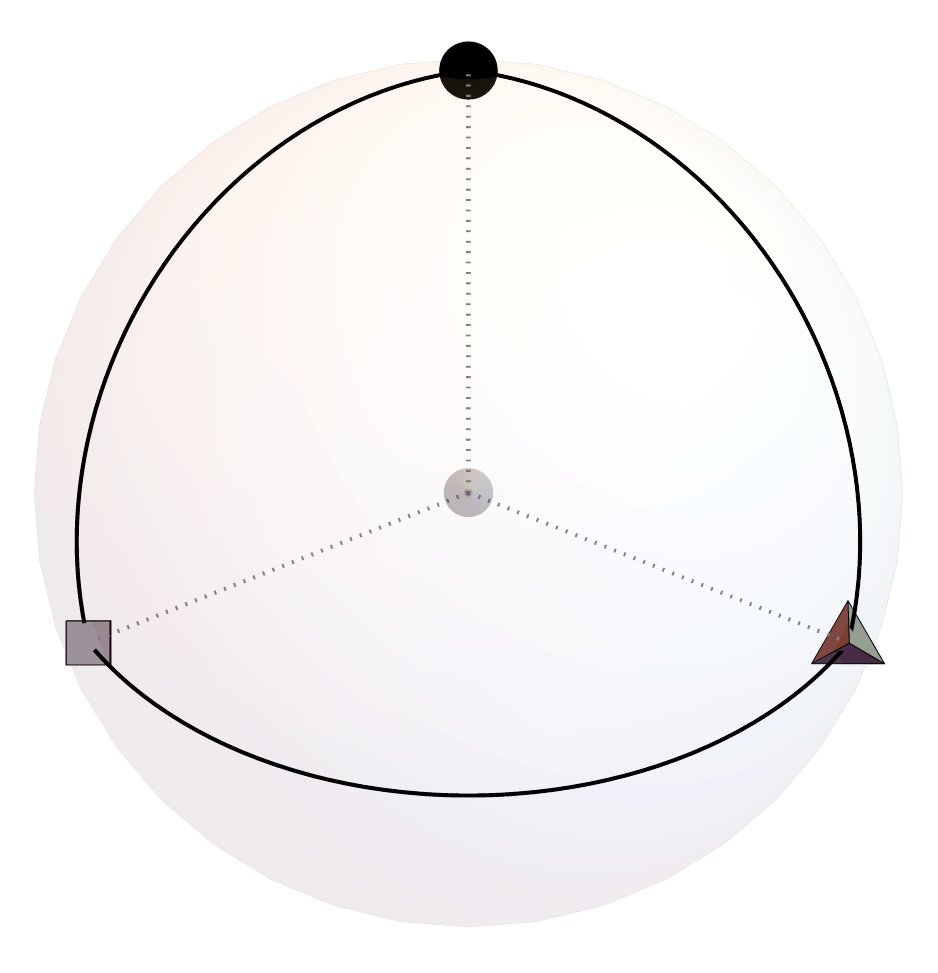}
   \caption{
   An example for  pair of equal mass isosceles $LRE$
   predicted by the symmetry
   $q(\sigma,\sigma_{12})=0 \Leftrightarrow q(\pi-\sigma,\pi-\sigma_{12})=0$
   seen from above the North pole.
   The left is $\sigma_{12}=\pi/3$ and $\sigma_{23}=\sigma_{31}=1.33240...$.
   The right is $\sigma_{12}=2\pi/3$ and 
   $\sigma_{23}=\sigma_{31}=\pi-1.33240...=1.80918...$.
   The grey ball at the centre represents the North pole.
   They have the common angular velocity $\omega^2=3.85072...$.}
   \label{figisoscelesSigma12EqPiDiv3and2PiDiv3}
\end{figure}
It is obvious that if $(\sigma, \sigma_{12})$ is a solution of 
$q(\sigma,\sigma_{12})=0$,
then $(\pi-\sigma, \pi-\sigma_{12})$ is also a solution.
Namely, $q(\sigma,\sigma_{12})=0$ is point symmetric
around $(\sigma,\sigma_{12})=(\pi/2,\pi/2)$.
Since $U'(\cos\sigma_{ij})=1/\sin^3(\sigma_{ij})$,
the angular velocity $\omega^2$ given by \eqref{omega} is invariant by this symmetry.
See Figure \ref{figisoscelesSigma12EqPiDiv3and2PiDiv3}.

In Figure \ref{figIsoscelesSigma12EqPiDiv6} we show three isosceles $LRE$ with $\sigma_{12}=\pi/6$.

 \begin{figure}
    \centering
\includegraphics[width=3cm]{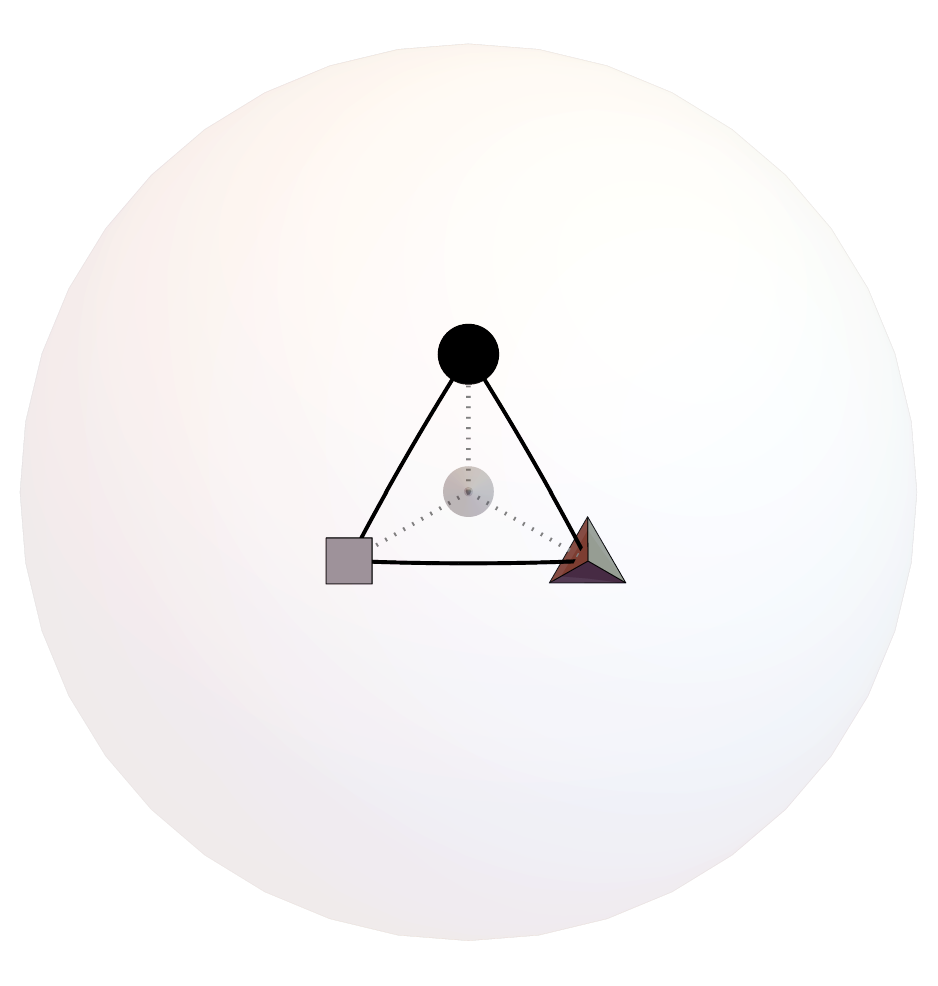}
\includegraphics[width=3cm]{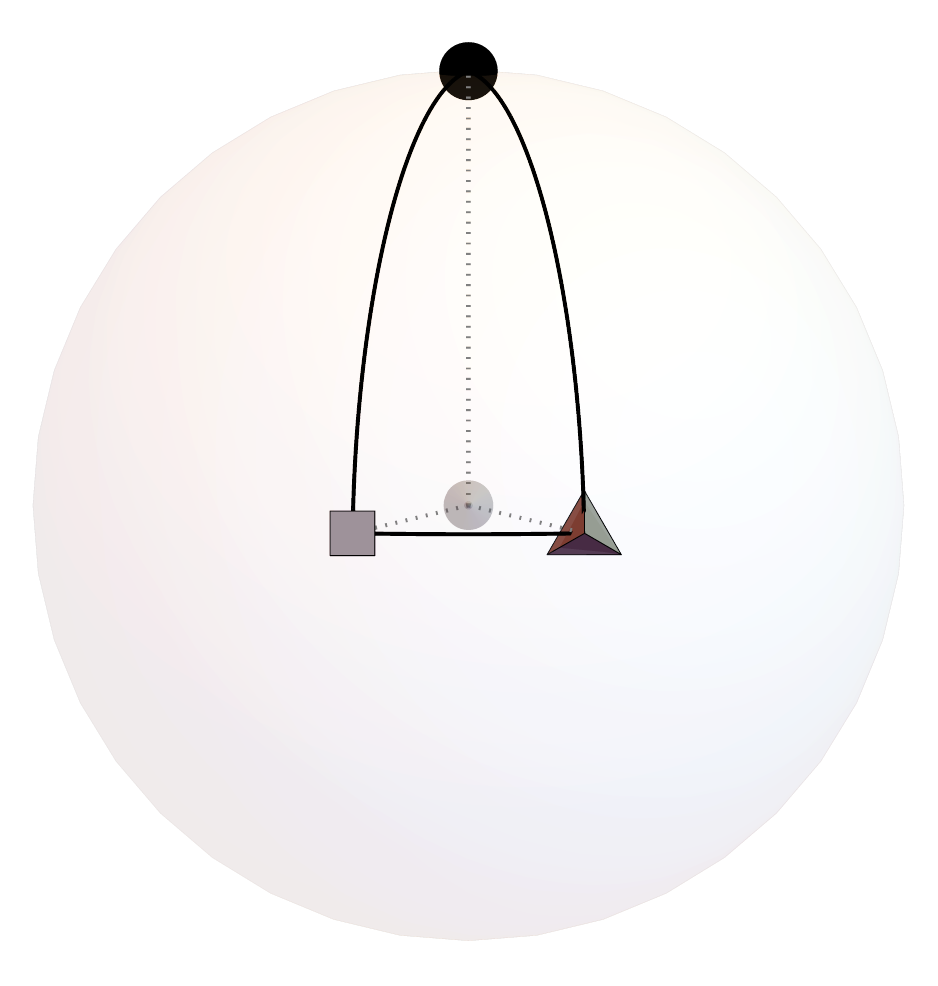}
\includegraphics[width=3cm]{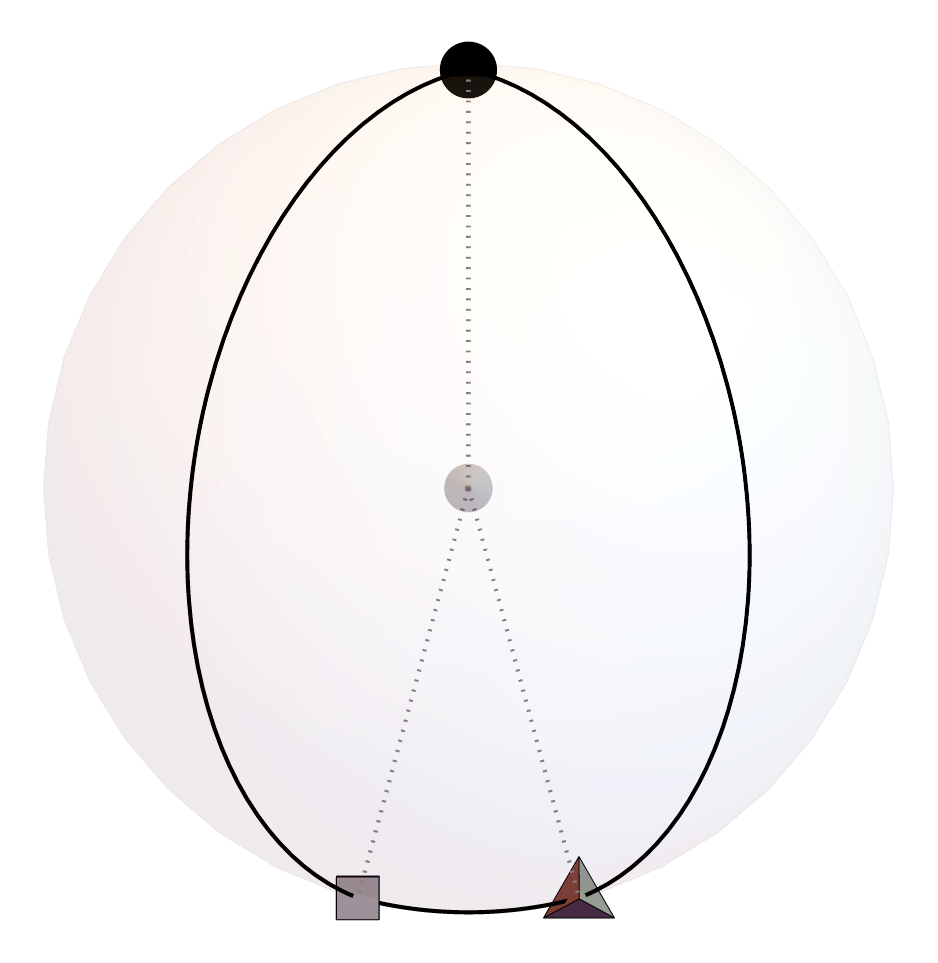} 
    \caption{Three isosceles Lagrangian $RE$ configurations
    with $\sigma_{12}=\pi/6$
    for equal masses,
    seen from above the North pole.
    From left to 
    right
    $\sigma_{23}=\sigma_{31}=
    \pi/6$ (equilateral triangle), $1.51596...\leq \pi/2 $, 
    and $2.73083...\geq 5\pi/6$.}
    \label{figIsoscelesSigma12EqPiDiv6}
 \end{figure}

\section{Conclusions and final remarks}\label{conclusions}
We successfully derive the conditions for general shapes of 
three bodies on the sphere $\mathbb{S}^2$ to generate a relative equilibria. 

It is well known that
the rotation axis for $RE$ is one of the eigenvectors
of the inertia tensor.
We divide the analysis of relative equilibria on the sphere into two big classes, collinear ($ERE$) and non-collinear ($LRE$). 
In both cases, for $n=3$, we give the necessary and sufficient conditions to obtain a $RE$ for $\{m_k,\cos\sigma_{ij}\}$.

We introduced the 3 by 3 matrix $J$ which is similar to the inertia tensor $I$.
The usefulness of $J$ for the investigation of $LRE$ should be remarked.

Without $J$, first we must diagonalise $I$
then calculate $\theta_k$ for each eigenvector.
The characteristic equation to obtain the eigenvalues is of third order. Then, we must find the eigenvector for each eigenvalue, and
finally, we must verify if the corresponding configuration satisfy the equations of motion or not. This procedure, if we manage to finish it, can be long and tedious.

Instead of doing that, by using the matrix $J$ that we found, the problem to obtain $LRE$
is reduced to solve the eigenvalue problem,
$J\Psi_L=\lambda\Psi_L$,
which directly gives the condition 
for $\{m_k,\cos\sigma_{ij}\}$
to form a $LRE$.
Thus, by using this 
condition, 
we give a complete answer to the question
described in the introduction.

To show how the condition 
works to determine $RE$ on $\mathbb{S}^2$,
we study the equal masses case for the cotangent potential.
The $ERE$ are completely determined, and the $LRE$ are almost.
The remaining problem is whether scalene $LRE$ exist or not.
The answer will be published later in a forthcoming paper,
with our results of investigation of $LRE$ for general masses.

So, the necessary and sufficient condition $J\Psi_L=\lambda\Psi_L$
for $\{m_k, \cos\sigma_{ij}\}$ works very fine.

The structure of the search of this condition
is the following: We derive Theorems 
\ref{threeEquivalentStatements}, 
\ref{PsitoEz}
and \ref{theCoditionInSigmaAndTheta}, each one describes 
the condition in terms of $m_k$, $\cos\sigma_{ij}$
and $\cos\theta_k$.
Theorems \ref{threeEquivalentStatements}, and 
\ref{PsitoEz} 
come from the considerations for
the conservation of the angular momentum and the inertia tensor,
Theorem \ref{theCoditionInSigmaAndTheta} comes directly from  the equations of motion.
Eliminating 
$\cos\theta_k$ 
from the above theorems,
we finally obtain the condition
in  Theorem \ref{ConditonForLRE}.

The last comment
is for the analysis of the $ERE$ 
on a rotating meridian
for general masses on the sphere.
The condition $det=0$
is a natural extension of  the Euler's fifth order equation.
For the cotangent potential, we have

\begin{equation*}
\begin{split}
&det=\frac{m_1m_2m_3 \,g}
{\sin\theta_{12}\sin\theta_{23}\sin\theta_{31}
|\sin\theta_{12}\sin\theta_{23}\sin\theta_{31}|},
\mbox{ where}\\
&g=\sum m_k \sin\theta_{ij}|\sin\theta_{ij}|
\big(\sin\theta_{ki}|\sin\theta_{ki}|\sin(2\theta_{ki})
-\sin\theta_{jk}|\sin\theta_{jk}|\sin(2\theta_{jk})
\big),
\end{split}
\end{equation*}
the 
sum runs for $(i,j,k)\in cr(1,2,3)$.
Note that $g=0$ is a fifth order equation
for $\sin\theta$.
We can easily verify that 
 limit $R\to \infty$ with $R\, |\theta_{ij}|=r_{ij}$ fixed,
yields the Euler's fifth order 
equation for $r_{ij}$.

\subsection*{Acknowledgements}
Thanks to our friend Florin Diacu, 
who was the inspiration of this work.
The second author (EPC) has been partially supported 
by Asociaci\'on Mexicana de Cultura A.C. and Conacyt-M\'exico Project A1S10112.

\end{document}